\documentclass[11pt,a4paper]{amsart}

\usepackage[OT4]{fontenc}
\usepackage{amsthm}
\usepackage{amsmath}
\usepackage{amsfonts}
\usepackage{graphicx}
\usepackage[margin=20mm]{geometry}

\newtheorem{tw}{Theorem}[section]
\newtheorem{cor}[tw]{Corollary}
\newtheorem{lem}[tw]{Lemma}
\newtheorem{df}[tw]{Definition}
\newtheorem{claim}[tw]{Claim}

\newtheorem*{twA}{Theorem A}
\newtheorem*{twB}{Theorem B}

\begin{document}

\begin{abstract}
We prove the analogue of Helly's theorem for systolic complexes. Namely, we show that $7$-systolic complexes have Helly dimension 
less or equal to $1$, whereas $6$-systolic complexes have Helly dimension bounded from the above by $2$.
\end{abstract}

\author{Krzysztof {\'S}wi\k{e}cicki}
\address{Department of Mathematics, Texas A\&M University, MS-3368, College Station, TX 77843-3368, USA}
\email{ksas@math.tamu.edu}
\title{Helly's theorem for systolic complexes}

\maketitle

\section{Introduction}
Eduard Helly proved his classical theorem concerning convex subsets of Euclidean spaces. Namely, 
suppose that $X_1,X_2,\dots,X_n$ is a collection of convex subsets of $\mathbb{R}^d$ (where $n > d$)
such that the intersection of every $d+1$ of these sets in nonempty. Then the whole family has a nonempty intersection.

This result gave rise to the concept of Helly dimension.
For a geodesic metric space $X$ we define its \emph{Helly dimension} $h(X)$ to be the smallest natural number such that any 
finite family of $(h(X)+1)$-wise non-disjoint convex subsets of $X$ has a non-empty intersection. 
Clearly, Helly's theorem states that Helly dimension of the Euclidean space $\mathbb{R}^d$ is $ \leq d$.
It is very easy to find examples showing that it is exactly equal to $d$.

Mikhail Gromov in \cite{Gro} gave a purely combinatorial characterization of $CAT(0)$ cube complexes as simply connected cube 
complexes in which the links of vertices are simplicial flag complexes. There is a well known result for $CAT(0)$ cube 
complexes which states that, regardless their topological dimension, they all have  Helly dimension equal to one 
(see \cite{R}). 
Note that in this case we additionally demand that convex subsets from the definition are convex subcomplexes.

Systolic complexes were introduced by Tadeusz Januszkiewicz and Jacek {\'S}wi\k{a}tkowski in \cite{JS1} and independently 
by Frederic Haglund in \cite{sys2}. They are connected, simply connected simplicial complexes satisfying some additional 
local combinatorial condition (see Definition \ref{sysdef}), which is a simplicial analogue of nonpositive curvature. Systolic 
complexes inherit lots of $CAT(0)$-like properties, however being systolic neither implies, nor is implied by nonpositive 
curvature of the complex equipped with the standard piecewise euclidean metric.

One would expect for systolic complexes a similar kind of Helly-like properties as for $CAT(0)$ cube complexes. 
However, let us consider an example of a single $n$-dimensional simplex with its codimension $1$ faces being a family 
of convex subcomplexes. Then the intersection of any subfamily of the cardinality $n$ is non-empty, but the intersection 
of the entire family is empty.
This example motivates the following modification of definition of Helly dimension for systolic complexes.
Namely, we say that a systolic complex $X$ has Helly dimension $h(X)\leq d$ if for every $(d+1)$-wise family of
non-disjoint convex subcomplexes there is a single simplex which has a nonempty intersection with all these
subcomplexes.
In this paper we prove that Helly dimension of a systolic complex does not depend on its topological dimension. 
We obtained the following results:

\begin{twA}(see Theorem \ref{tw7} in this text)
Let $X$ be a $7$-systolic complex and let $X_1, X_2, X_3$ be pairwise intersecting convex subcomplexes. Then there
exists a simplex $\sigma \subseteq X$ such that $\sigma \cap X_i \neq \emptyset$ for $i = 1,2,3$.
Moreover, $\sigma$ can be chosen to have the dimension at most two.
\end{twA}

In other words $7$-systolic complexes have Helly dimension less or equal to $1$. 
It is easy to see that this is not necessarily true for $6$-systolic complexes (see Figure \ref{greater}), 
but we prove that any systolic complex has Helly dimension less or equal to $2$. More precisely:

\begin{figure}[ht!]
\begin{center}
\includegraphics[width=0.25\textwidth]{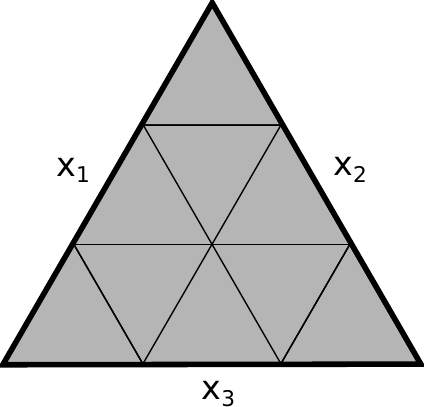}
\caption{A systolic complex with Helly dimension $>1$.}\label{greater}
\end{center}
\end{figure}

\begin{twB}(see Theorem \ref{tw6} in this text)
Let $X$ be a systolic complex and let $X_1, X_2, X_3, X_4$ be its convex subcomplexes such that every three of
them have a nonempty intersection. Then there exists a simplex $\sigma \subseteq X$ such that 
$\sigma \cap X_i \neq \emptyset$ for $i = 1,2,3,4$.
Moreover, $\sigma$ can be chosen to have the dimension at most three.
\end{twB}

\newpage

\textbf{Acknowledgements}. 
I am grateful to my advisor Pawe{\l} Zawi{\'s}lak for his constant support, patience and many interesting discussions.
Without his help neither this thesis would happen, nor I would be in the place that I am today.
I also want to thank Jacek {\'S}wi\k{a}tkowski for suggesting the topic of this thesis and many useful corrections and hints how 
to improve this paper.

\section{Preliminaries} 

In this section we recall basic definitions and some results concerning systolic complexes.

A \emph{simplicial complex} $X$ is a set of simplices that satisfy the following conditions. Any face of a simplex in $X$ 
is also in $X$. The intersection of two simplices of $X$ is a face of both of them.
A simplicial complex $X$ is \emph{flag} if any finite set of vertices in $X$ that are pairwise connected by edges spans 
a simplex of $X$.

For any collection $S$ of simplices in a simplicial complex $X$ we define the following.
The \emph{closure} of $S$, denoted by $Cl_X(S)$, is the minimal subcomplex of $X$ that contains every simplex of $S$.
The \emph{star} of $S$, denoted by $St_X(S)$, is the set of all simplices in $X$ that contain some face of a simplex of $S$.
The \emph{link} of $S$, denoted by $lk_{X}(S)$, equals $Cl_X( St_X( S )) - St_X( Cl_X (S))$.

A \emph{cycle} in $X$ is a subcomplex $\gamma$ isomorphic to some triangulation of $S^1$. We define the \emph{length}
$|\gamma|$ of $\gamma$ to be the number of its edges. A \emph{diagonal} of a cycle is an edge connecting its two 
nonconsecutive vertices.

\begin{df}\label{sysdef}(see Definition $1.1$ in \cite{JS1})
 Given a natural number $k \geq 4$ a simplicial complex $X$ is called:
 \begin{itemize}
  \item \emph{$k$-large} if it is flag and every cycle in $X$ of length $3< |\gamma| < k$ has a diagonal,
  \item \emph{locally $k$-large} if links of all simplices in $X$ are $k$-large,
  \item \emph{$k$-systolic} if it is connected, simply connected and locally $k$-large.
 \end{itemize}
\end{df}
One often calls locally $6$-largeness a simplicial nonpositive curvature and it is common to abbreviate $6$-systolic 
to systolic. As we already mentioned, systolic complexes inherit lots of $CAT(0)$-like properties. For example: 
they are contractible (\cite{JS1}), the analogue of the Flat Torus Theorem holds for them (\cite{E1}) and every finite
group acting geometrically on a systolic complex by simplicial automorphisms has a global fix point (\cite{fix}).
Moreover, $7$-systolic complexes are $\delta$-hyperbolic (\cite{JS1}).

Recall that a \emph{geodesic metric space} is a metric space in which every two points can be connected by an 
isometrically embedded segment called \emph{geodesic}.
Clearly, the $1$-skeleton of a systolic complex equipped with the standard combinatorial metric 
(i.e. all edges have length $1$) is a geodesic metric space.
It is important to clarify that we consider only geodesics contained in $1$-skeleton, with both endpoints in the 
$0$-skeleton of a complex.
Thus we can identify a combinatorial geodesic with a sequence of vertices contained in it.
We denote by $( v_0, v_1, \ldots v_n )$ a geodesic starting at $v_0$, passing through vertices $v_1, \ldots v_{n-1}$
and terminating at $v_n$. We also denote by $(v,w)$ a geodesic from $v$ to $w$.

A subcomplex $A$ of a simplicial complex $X$ is \emph{geodesically convex} if for any two vertices 
$x,y \in A$, $A$ contains every shortest path in $1$-skeleton of $X$ between $x$ and $y$.
A subcomplex $A$ of a simplicial complex $X$ is called \emph{$3$-convex} if it is full and any combinatorial geodesic 
$\gamma$ of length $2$ with both endpoints in $A$ is entirely contained in $A$.
A subcomplex $A$ of a systolic complex $X$ is \emph{convex} if it is connected and \emph{locally $3$-convex}, 
i.e. for every every simplex $\sigma \in A$, $lk_{A}(\sigma)$ is $3$-convex in $lk_{X}(\sigma)$. 
It turns out that a subcomplex of systolic complex is convex iff it is geodesically 
convex (see Proposition $4.9$ in \cite{hs}). 

We now recall an important tool often used in the study of systolic complexes. 
We start with some definitions. Let $M$ be a triangulation of a 2-dimensional manifold.
Let $v$ be a vertex of $M$ and let $\chi(v)$ be a number of triangles containing $v$.
A \emph{defect} of vertex $v$ (denoted by $def(v)$) is equal to $6-\chi(v)$ for the interior
vertices and to $3-\chi(v)$ for the boundary vertices. A vertex $v$ is called \emph{negative} 
(respectively \emph{positive}) if $def(v)<0$ (respectively $def(v)>0$).

\begin{lem}(combinatorial Gauss-Bonnet)\label{GB}
 Let $M$ be a triangulation of a 2-dimensional manifold. Then:
 \begin{displaymath}
   \sum_{v\in{\mathrm{M}}}def(v)=6\chi(M),\ 
 \end{displaymath}
 where $\chi(M)$ denotes the Euler characteristic of $M$.
\end{lem}

Let now $X$ be a systolic complex. Any simplicial map $S:\Delta_S \to X$, where $\Delta_S$ is a triangulation of a 
$2$-disc, will be called a \emph{surface}. We say that a surface $S$
is \emph{spanned} by a cycle $\gamma$ if it maps $\partial \Delta_S$ isomorphically onto $\gamma$. By an \emph{area} 
of a simplicial disc we mean the number of its triangles. Similarly, an \emph{area} of a surface $S$
is the number of triangles of $\Delta_S$ on which $S$ is injective. 
If $\Delta_S$ has the minimal area among surfaces extending map $\partial \Delta_S \to X$ we call both the surface and 
the map associated to it \emph{minimal}. 
We say that a \emph{surface $S$ is systolic} if $\Delta_S$ is systolic.
The existence of minimal surfaces is given by the next lemma. Note that the original proof of Lemma $4.2$ in \cite{E1}
deals only with the systolic case. However, the same reasoning applies for the $k$-systolic case.

\begin{lem}(see Lemma $4.2$ in \cite{E1})\label{4.2}
 Let $X$ be a $k$-systolic complex and let $S^1$ be a triangulated circle. Then any simplicial map $f: S^1 \to X$ can be 
 extended to a simplicial map $F:\Delta \to X$, where $\Delta$ is a $k$-systolic disc such that $\partial \Delta = S^1$. 
 Moreover, any minimal surface extending $f$ is $k$-systolic.
\end{lem}

A surface $S:\Delta_S \to X$ is called \emph{flat} if $\Delta_S$ is \emph{flat disc}, i.e. $\Delta_S^{(1)}$
can be isometrically embedded into $1$-skeleton of the equilaterally triangulated Euclidean 
plane $\mathbb{R}^2_{\Delta}$. We have the following easy characterization of flatness:

\begin{lem}(see Theorem $3.5$ in \cite{E1})\label{flatkus}
 A simplicial disc $\Delta$ is flat if and only if it satisfies the following three conditions:
 \begin{enumerate}
  \item every internal vertex of $\Delta$ has defect $0$,
  \item $\Delta$ has no boundary vertices of defect less than $-1$,
  \item on $\partial \Delta$ any two negative vertices are separated by a positive one.
 \end{enumerate}
\end{lem}

We now focus on geodesic triangles in systolic complexes. First, we recall the following:

\begin{lem}(see Lemma $4.3$ in \cite{E2})\label{4.3}
 If $x_0, x_1, x_2$ are vertices of a systolic complex $X$, then for $i = 0, 1, 2$ there exist geodesics $\gamma_i$  with 
 endpoints $x_{i-1}$ and $x_{i+1}$ (we use the cyclic order of indices) such that:
 \begin{enumerate}
  \item $\gamma_i \cap \gamma_{i+1}$ is a geodesic (possibly degenerated) with endpoints $x_{i-1}$ and $x'_{i-1}$
  \item if we denote the subgeodesic (possibly degenerated) with endpoints $x'_{i-1}$ and $x'_{i+1}$ by 
  $\gamma'_i \subset \gamma_i$, then either $x'_0=x'_1=x'_2$ or a minimal surface $S: \Delta_S \to X$ spanning the cycle 
  $\gamma'_0 \ast \gamma'_1 \ast \gamma'_2$ has an equilaterally triangulated equilateral triangle as the domain.
 \end{enumerate}
\end{lem}

\begin{figure}[ht!]
\begin{center}
\includegraphics[width=0.29\textwidth]{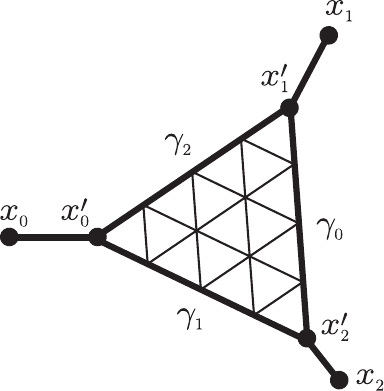}
\caption{A domain of a triangular surface.}\label{rownoboczny}
\end{center}
\end{figure}

Before we move forward, we introduce some terminology.
Let $z'_1, z'_2, \ldots z'_n \in S^1$ be cyclically ordered different points on a circle. 
A \emph{disc with $n$ horns} is a quotient space $D^2 \bigsqcup {\{1,2,\ldots n\} \times [0,1] }/_\sim$ where
$\sim$ is an equivalence relation given by $z'_i \sim (i,0)$. Points $x'_i=[z'_i]_\sim$ are called \emph{bottoms} of 
the horn and points $x_i=[(i,1)]_\sim$ are called \emph{tops} of the horn.
For every $i=1,2, \ldots n$ there is a distinguished segment from $x_i$ to $x_{i+1}$, namely
$[x_i,x_{i+1}]=[x_i,x'_i] \cup [x'_i,x'_{i+1}] \cup [x'_{i+1},x_{i+1}]$, where $[x'_i,x'_{i+1}]$ is an arc 
in $S^1$ connecting bottoms $x'_i$ and $x'_{i+1}$.
As a corollary of Lemma \ref{4.3} we have the following:

\begin{cor}\label{generalus}
 Suppose that $v_0, v_1, v_2$ are three distinct vertices in a systolic complex $X$. Then there exist a simplicial map
 $S_{v_0, v_1, v_2} : T_{v_0, v_1, v_2} \to X$ satisfying one of the following conditions:
 \begin{enumerate}
  \item $T_{v_0, v_1, v_2}$ is a segment and $S_{v_0, v_1, v_2}$ is an isometry (this holds iff one of the points 
  $v_0, v_1, v_2$ lies on some geodesic connecting other two of them).
  \item $T_{v_0, v_1, v_2}$ is a tripod and $S_{v_0, v_1, v_2}$ is an isometry (this holds iff there exist geodesics
  $\gamma_1,\gamma_2,\gamma_3$ connecting $v_0, v_1, v_2$ such that $\gamma_1 \cap \gamma_2 \cap \gamma_3  \neq \emptyset$)
  \item $T_{v_0, v_1, v_2}$ is a equilaterally triangulated equilateral disc with $\leq 3$ horns and
  $S_{v_0, v_1, v_2}$ maps the segment $[x_i,x_{i+1}]$ isometrically onto some geodesic $[v_i,v_{i+1}]$.
 \end{enumerate}
\end{cor}

A map $S_{v_0, v_1, v_2}$ is called \emph{triangular surface} spanned on vertices $v_0, v_1, v_2$. It is called 
minimal if $1$ or $2$ holds, or in the case $3$ if $T_{v_0, v_1, v_2}$ has minimal area among domains of 
triangular surfaces spanned on $v_0,v_1,v_2$.
A geodesic triangle between $v_0, v_1, v_2$ is called a \emph{minimal geodesic triangle} if boundary of a minimal 
surface is mapped onto it.

Systolic complexes are contractible. Therefore if $X$ is systolic then every map $f:S^n \to X$ can be extended to a map
$F:B^{n+1}\to X$. Similarly as for circles, $2$-dimensional spheres have fillings with special properties. Precisely,
we have the following result:

\begin{lem}(see Theorem $2.5$ in \cite{E1})\label{4.4}
Let $X$ be a systolic complex and let $S$ be a triangulation of a $2$-sphere. Then any simplicial map $f:S \to X$ can be 
extended to a simplicial map $F:B \to X$, where $B$ is a triangulation of a 3-ball such that $\partial B = S$ and
$B$ has no internal vertices.
\end{lem}

\section{$7$-systolic case}
We start with the proof of a useful lemma in slightly more general setting:
\begin{lem}\label{wlozenie}
 Let $X_1, X_2, X_3$ be three pairwise intersecting geodesically convex subcomplexes in a connected simplicial complex $X$. Then we 
 can pick a  triple of points $A,B,C$ such that $A \in X_1 \cap X_2$, $B \in X_1 \cap X_3$, $C \in X_2 \cap X_3$ and 
 geodesics  $\gamma_1 \subset X_1$, $\gamma_2 \subset X_2$, $\gamma_3 \subset X_3$ between them in such a way that 
 either a  geodesic triangle $\gamma_1 \ast \gamma_2 \ast \gamma_3$ is isomorphic to a triangulation of $S^1$ or
 $\gamma_1 \cap \gamma_2 \cap \gamma_3  \neq \emptyset$.
\end{lem}

\begin{proof}
 Let $A \in X_1 \cap X_2$, $B \in X_1 \cap X_3$ and $C \in X_2 \cap X_3$. 
 Since $X_1, X_2, X_3$ are convex, it follows that there are geodesics $\gamma_1 \subset X_1$, $\gamma_2 \subset X_2$, 
 $\gamma_3 \subset X_3$ connecting $A$ with $B$, $B$ with $C$ and $C$ with $A$ respectively.
 If either all of these geodesics have a nonempty intersection or they only intersect at endpoints, Lemma \ref{wlozenie} 
 is proved.
 
 \begin{figure}[ht!]
  \begin{center}
   \includegraphics[width=0.441\textwidth]{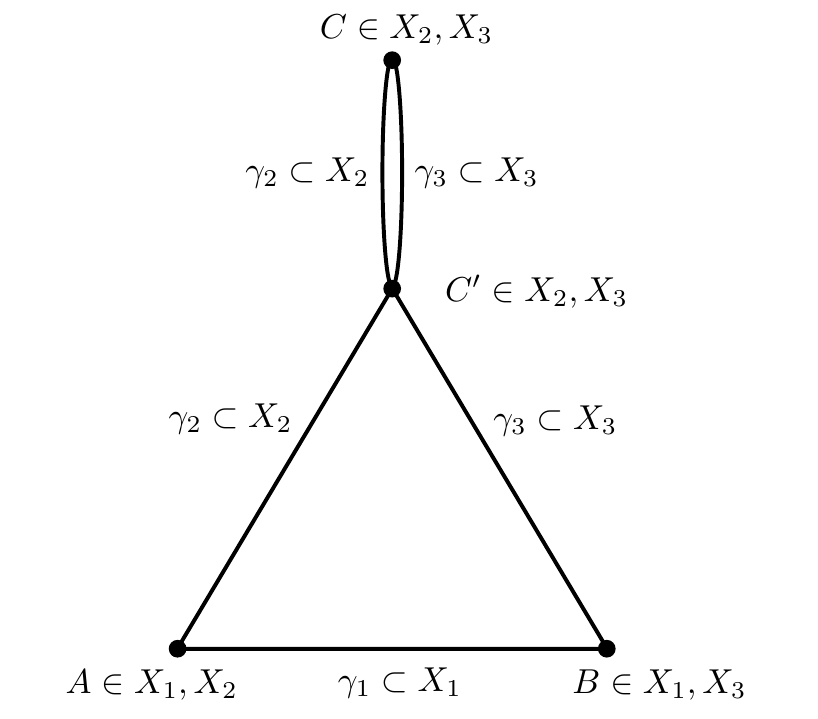}
   \caption{We can choose such $A,B,C$ that $\gamma_1 \ast \gamma_2 \ast \gamma_3 \simeq S^1$.}\label{7sys1}
  \end{center}
 \end{figure}
 
 Suppose now that $\gamma_1 \cap \gamma_2 \cap \gamma_3  = \emptyset$ and that some pair of these geodesics ,
 say $\gamma_2$ and $\gamma_3$, have a nonempty intersection besides endpoints. Let $C'$ be a vertex in this intersection.
 Since $C' \in \gamma_2 \cap \gamma_3 \subset X_2 \cap X_3$ we can replace point $C$ by the point $C'$ and
 obtain a new triple of points $A, B, C'$ with the same properties (see Figure \ref{7sys1} as an illustration).
 We repeat this procedure until $\gamma_1, \gamma_2, \gamma_3$ have no intersection except for the endpoints, hence until
 $\gamma_1 \ast \gamma_2 \ast \gamma_3$ is isomorphic to some triangulation of $S^1$.
 
\end{proof}

Now we can formulate and prove Helly's theorem for $7$-systolic complexes:

\begin{tw}\label{tw7}
Let $X$ be a $7$-systolic complex and let $X_1, X_2, X_3$ be pairwise intersecting convex subcomplexes. Then there
exists a simplex $\sigma \subseteq X$ such that $\sigma \cap X_i \neq \emptyset$ for $i = 1,2,3$.
Moreover, $\sigma$ can be chosen to have the dimension at most two.
\end{tw}

\begin{proof}
 Due to Lemma \ref{wlozenie} we can choose a triple of points $A \in X_1 \cap X_2$, $B \in X_1 \cap X_3$, 
 $C \in X_2 \cap X_3$ and geodesics $\gamma_1$, $\gamma_2$ and $\gamma_3$, between them such that
 $\gamma_1 \cap \gamma_2 \cap \gamma_3  \neq \emptyset$ or $\gamma_1 \ast \gamma_2 \ast \gamma_3$ is a cycle. 
 In the first case the theorem is proved. 

 Otherwise we choose these geodesics in such a way that $\gamma_1, \gamma_2, \gamma_3$ is a minimal geodesic triangle. 
 We know by Lemma \ref{4.2} that a minimal surface $S$ spanned by this cycle is $7$-systolic and
 since $ \chi( \Delta_S )=1$ by Lemma \ref{GB} we have:
 
 \begin{displaymath}
   \sum_{v\in{\mathrm{\Delta_S}}}def(v)=6.\ 
 \end{displaymath}
 
 Since $\Delta_S$ is $7$-systolic, the link of its every interior vertex is a cycle of length at least 7 and hence 
 $def(v) < 0$ for every vertex $v \in{\mathrm{int}}{\Delta_S}$.
 The preimages $A', B', C' \in \Delta_S$ by $S$ of $A, B$ and $C$ respectively belong to minimum one triangle, 
 so the total contribution to the left hand side of  the above equation from those points is at most $6$.
 
 Consider now a geodesic $\gamma_1=(v_0,v_1,\ldots ,v_l)$, where $l$ denotes its length, $v_0=A$, $v_l=B$ 
 and set $v_i'=S^{-1}(v_i)$. We will show that $(def(v'_i)) \leq 0$ for $i = 1 \ldots l-1$, i.e. that every 
 such vertex must be contained in at least three triangles.

 Indeed, if $v'_i$ belongs to one triangle, there is an edge connecting $v'_{i-1}$ and $v'_{i+1}$ and
 the same holds for $v_{i-1}$ and $v_{i+1}$. This contradicts the fact that $\gamma_i$ is a geodesic 
 (see Figure \ref{7sys2}). 
 
 Similarly, if $v'_i$ belongs to two triangles, let $w'$ be their common vertex different from $v'_i$.
 Set $w=S(w')$ and note that there is another geodesic 
 $\gamma_1^{w}=(v_0, \ldots v_{i-1}, w, v_{i+1},\ldots, v_l)$ connecting $A$ to $B$ (see Figure \ref{7sys2}).
 Note also that the geodesic triangle ($\gamma_1^{w},\gamma_2,\gamma_3$) has a filling with smaller area than 
 ($\gamma_1,\gamma_2,\gamma_3$) (a restriction of the map $S$ to the disc $\Delta_S$ with simplices 
 $(v'_{i-1}, v'_i, w')$ and $(v'_i,v'_{i+1}, w')$) removed. This contradicts the fact that 
($\gamma_1,\gamma_2,\gamma_3$) has a minimal area.  
 We repeat the same reasoning for geodesics $\gamma_2,\gamma_3$.
 
  \begin{figure}[ht!]
  \begin{center}
   \includegraphics[width=0.925\textwidth]{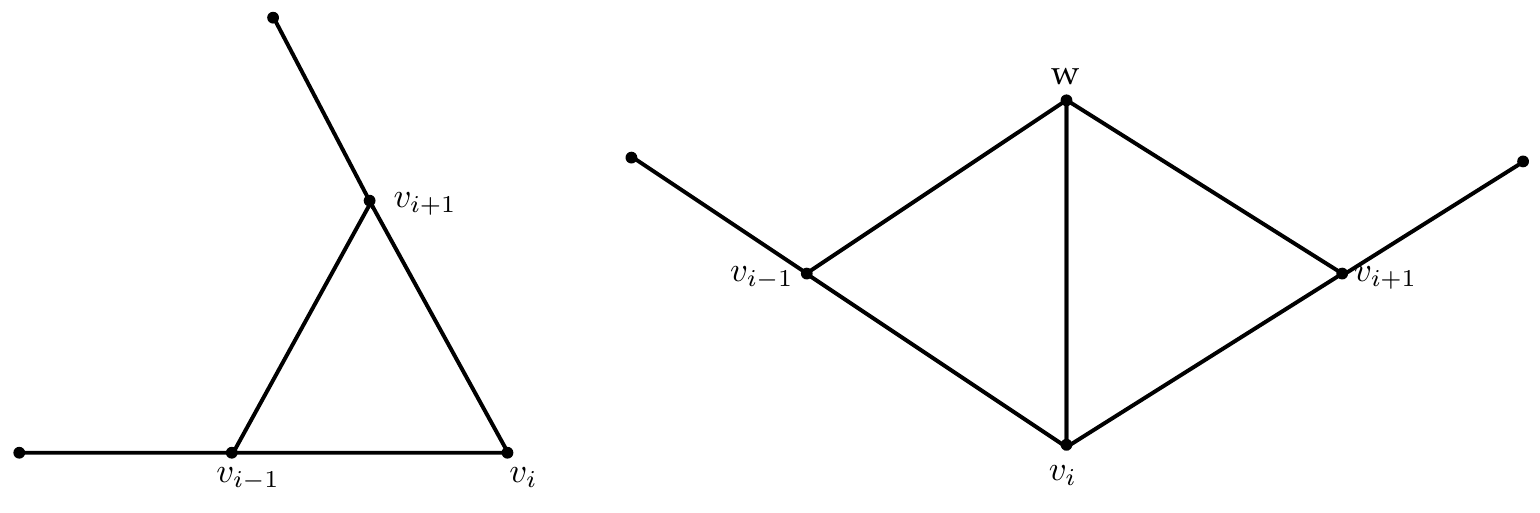}
   \caption{The vertex $v_i$ must be contained in at least three triangles.}\label{7sys2}
  \end{center}
 \end{figure}
 
 We conclude that if $\Delta_S$ is a domain of a minimal surface spanned by a minimal geodesic triangle, in order to 
 satisfy Gauss-Bonnet Lemma, it cannot contain any interior vertex. Vertices $A', B', C'$ have defect $2$ and the 
 rest of boundary vertices have defect $0$. Therefore due to Lemma \ref{flatkus} $\Delta_S$ is flat.
 It follows that $\Delta_S$ is a single $2$-simplex and its image under $S$ satisfies demanded properties.
 
\end{proof}

\begin{cor}
 Let $A,B,C$ be three distinguished points in a $7$-systolic complex $X$. Then a minimal geodesic triangle 
 between them consists of three geodesic segments and a $1$-skeleton of the one (possibly degenerated) 
 $2$-simplex 
(see Figure \ref{7sys3}).
 
 \begin{figure}[ht!]
  \begin{center}
   \includegraphics[width=0.3\textwidth]{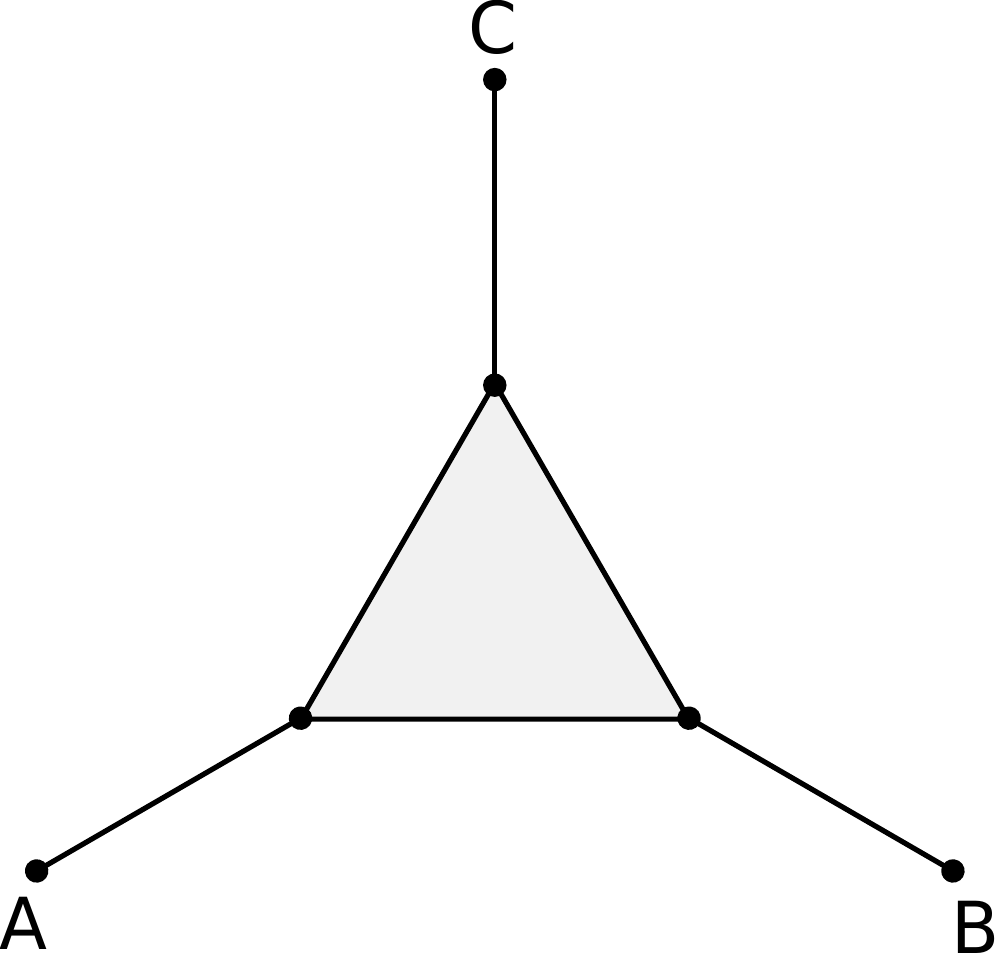}
   \caption{A shape of a minimal geodesic triangle in $7$-systolic complex.}\label{7sys3}
  \end{center}
 \end{figure}
\end{cor}

\section{$6$-systolic case}

We start with proving some useful lemmas and introducing some terminology.

\begin{lem}\label{digon}
 Let $x, y$ be two distinct points in a systolic complex $X$ and let $\gamma_0$ and $\gamma_1$ be two geodesics between 
 them. If these geodesics intersect only at endpoints, then a minimal surface spanned by $\gamma_0 \ast \gamma_1$ is flat.
\end{lem}

\begin{proof}
 Let $S:\Delta_S \to X$ be a minimal surface spanned by $\gamma_0 \ast \gamma_1$, which by Lemma \ref{4.2} is systolic.
 We need to check if $\Delta_S$ satisfies assumptions of Lemma \ref{flatkus}. 

 First let us consider the boundary vertices different from $S^{-1}(x)$ and $S^{-1}(y)$. 
 Since they lay on a geodesic, any two such positive vertices are separated by a negative vertex (see the proof of Theorem 
 $2.1$ in \cite{JS1}). Thus the total defect on each geodesic $\gamma_0$ and $\gamma_1$, excluding endpoints, is less or equal 
 to $1$. By Gauss-Bonnet Lemma we have:
 
 \begin{displaymath}
   \sum_{v\in{\mathrm{M}}}def(v)=6\chi(M) 
 \end{displaymath}
 
 Since $\Delta_S$ is systolic, the defect of every interior vertex is nonpositive. And since $\gamma_0$ and $\gamma_1$
 meet only at endpoints, vertices $S^{-1}(x)$ and $S^{-1}(y)$ have defect non greater then $2$. 
 In order to satisfy Gauss-Bonnet Lemma we see that every interior vertex has to have defect $0$ and on geodesics
 we have alternating vertices with defect $1$ and $-1$, possibly separated by vertices with defect $0$. 
 Therefore assumptions of Lemma \ref{flatkus} are satisfied and $\Delta_S$ is flat.
\end{proof}

A surface $S$ as in Lemma \ref{digon} is called a \emph{simple digon} spanned by $\gamma_0 \ast \gamma_1$.
Now we formulate, the following obvious consequence of Lemma \ref{digon}:

\begin{cor}
 Let $x, y$ be two distinct points in a systolic complex $X$ and let $\gamma_0$ and $\gamma_1$ be two geodesics
 connecting them. Then there exist a simplicial map $S_{x, y}^{\gamma_0, \gamma_1} : D_{x, y}^{\gamma_0, \gamma_1} 
 \to X$ satisfying one of the following conditions:
 \begin{enumerate}
  \item $D_{x, y}^{\gamma_0, \gamma_1}$ is a segment and $S_{x, y}^{\gamma_0, \gamma_1}$ is an isometry
  (this holds iff $\gamma_0 = \gamma_1$),
  \item $S_{x, y}^{\gamma_0, \gamma_1}$ is a simple digon,
  \item there is a sequence $D_1, D_2, \ldots D_n \subseteq D_{x, y}^{\gamma_0, \gamma_1}$ s.t. 
  $D_{x, y}^{\gamma_0, \gamma_1} = \bigcup_{i=1}^{n} D_i$, $D_i \cap D_j \neq \emptyset \iff \| i-j \| \leq 1$,
  $D_i \cap D_{i+1} = \{ v_i \} \in X^{(0)}$ and $ {S_{x, y}^{\gamma_0, \gamma_1}|}_{D_i}: D_i \to X$ 
  has a form as in $1$ or $2$.
 \end{enumerate}
\end{cor}

A map $S_{x, y}^{\gamma_0, \gamma_1}$ is called a \emph{digonal surface} spanned by $\gamma_0$, $\gamma_1$.
It is called \emph{minimal} if $D_{x, y}^{\gamma_0, \gamma_1}$ has minimal area among domains of digonal 
surfaces spanned by $\gamma_0$, $\gamma_1$.

\begin{lem}\label{conv}
 Let $S_{x,y,z}$ ($S_{x, y}^{\gamma_0, \gamma_1}$) be a minimal triangular (digonal) surface spanned
 on vertices $x, y, z$ (by $\gamma_0$, $\gamma_1$), which are contained in a convex subcomplex $A$ of a 
 systolic complex $X$. Then the image of the whole surface is contained in $A$.
\end{lem}

\begin{proof}
 First Consider $S:\Delta_S \to X$ being a simple digon spanned by geodesics $\gamma_0, \gamma_1$, which connect 
 $x, y$. Since $x, y \in A$ and $A$ is convex it follows that also $\gamma_0, \gamma_1 \subseteq A$.
 We use the induction over the area of $\Delta_S$. 
 The statement is true for two points in distance $1$ or $2$. 

 We proceed with the induction step.
 From the proof of Lemma \ref{digon} it follows that at least one of the vertices 
 on each geodesic part of $\partial \Delta_S$ has defect $1$ (see Figure \ref{6sysflat}). Denote this vertex $v_i$ and
 let $v_{i-1}$ and $v_{i+1}$ be its neighbours on this geodesic.
 Let $w$ be a second common vertex of two triangles containing $v_i$.
 Note that as $S(v_{i-1})$ and $S(v_{i+1})$ lie on a geodesic $d(S(v_{i-1}), S(v_{i+1})) = 2$ and thus 
 $(S(v_{i-1}),S(w),S(v_{i+1}))$ is a geodesic.
 We can replace a geodesic from $x$ to $y$ passing through $v_i$ by the one passing through $w$.
 If $w$ does not belong to the other geodesic of $\partial \Delta_S$ we obtain a simple digon with smaller area and we can 
 use the inductive assumption.
 If $w$ belongs to the other geodesic of $\partial \Delta_S$ we consider two simple digons: one between $x$ and $w$, second 
 between $w$ and $y$. They both have smaller area than original digon hence we can 
 use the inductive assumption to conclude the statement in this case.
 
 \begin{figure}[ht!]
  \begin{center}
   \includegraphics[width=0.5\textwidth]{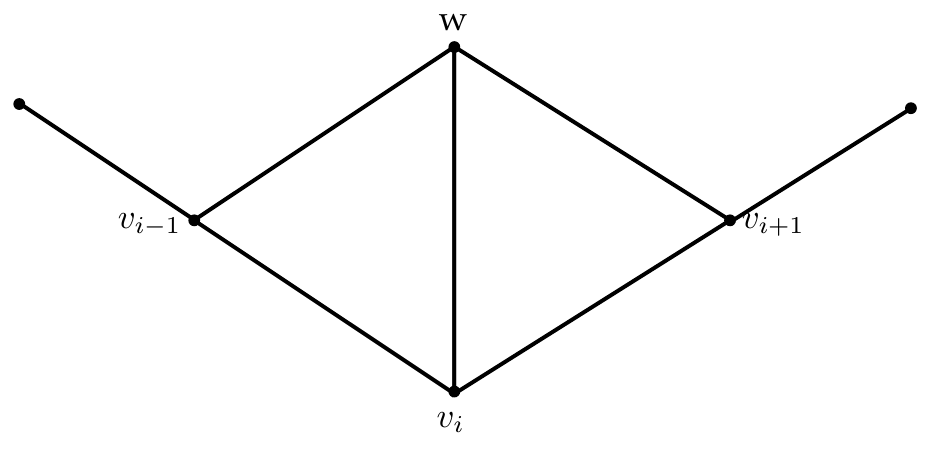}
   \caption{A vertex $v_i$ with  defect 1.}\label{6sysflat}
  \end{center}
 \end{figure}

 Now consider $S:\Delta_S \to X$ being a triangular surface spanned on $x, y, z$. 
 We use the induction over the length of a side of an equilateral triangle in
 $\Delta_S$. The statement is obvious for an equilateral triangle with side of length $1$. 

 We proceed with the induction step.
 Choose one side of the equilateral triangle in a $\Delta_S$ and denote its vertices by $v_{0}, v_1, \ldots v_l$. 
 Pick a maximal geodesic $(v',w_1,w_2, \ldots w_{l-2}, v'')$ in $\Delta_S$ consisting of vertices in distance $1$ from the side 
 $(v_{0},\ldots v_l)$ (see Figure \ref{6katwyp}).
 Notice that $S(v')$ and $S(v_i)$ for $i \in \{1 \ldots l\}$ belong to $A$. 
 Consider the image by $S$ of the path $(v',w_1,v_2)$. By Theorem  B in \cite{E1} applied to
 $\Delta_S$ with removed $2$-simplices containing preimages by $S$ of vertices $x, y, z$
 the restriction of $S$ to any simplicial disc contained in it with a diameter less or equal to three is an isometric embedding.
 Hence the path $(S(v'),S(w_1),S(v_2))$ is a geodesic in $X$.
 From convexity of $A$ it follows that also $S(w_1)$ belongs to $A$. 
 Repeating this argument we deduce that the image by $S$ of the entire geodesic $(v',w_1,w_2, \ldots w_{l-2}, v'')$ 
 is contained in $A$. It allows us to replace $\Delta_S$ with a smaller equilateral triangle and obtain the statement
 from the inductive assumption.
 
  \begin{figure}[ht!]
  \begin{center}
   \includegraphics[width=0.65\textwidth]{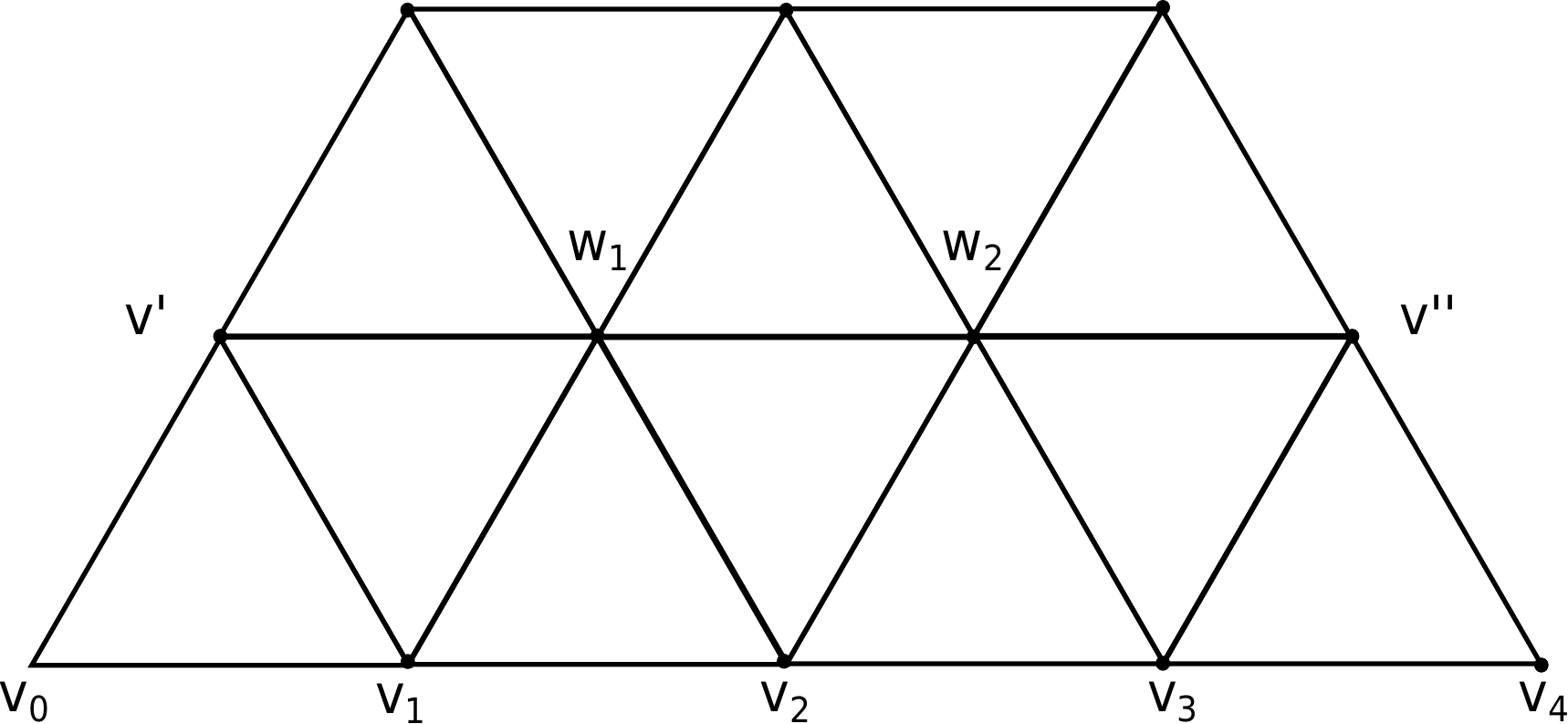}
   \caption{$S(w_1)$ belongs to $A$ from its $3$-convexity.}\label{6katwyp}
  \end{center}
 \end{figure}
 
\end{proof}

Now we construct a simplicial map that will be crucial in the proof of Helly's theorem.

\begin{lem}\label{spherus}
 Let $A, B, C, D$ be four distinct points in a systolic complex $X$. Then there exist:
 \begin{enumerate}
 \item a map $f_{A,B,C,D}: S \to X$, where $S$ is a triangulation of the $2$-sphere,
 \item an inclusion $i_{u,v,w}: T_{u,v,w} \to S$ of a domain of a minimal triangular surface for any 
 triple of different points $u, v, w \in \{A, B, C, D\}$,
 \item an inclusion $i_{u,v}: D_{u,w} \to S $ of a domain of some specified digonal surface 
 for every pair of different points $u, v \in \{A, B, C, D\}$.
 \end{enumerate}
 These maps satisfy the following conditions:
 \begin{enumerate}
 \item $f_{A,B,C,D} \circ i_{u,v,w} = S_{u,v,w}$,
 \item $f_{A,B,C,D} \circ i_{u,v}= S_{u,v}$.
 \end{enumerate}
  Moreover, images of digonal and triangular surfaces after taking the union form entire sphere $S$,
  and are disjoint except the preimages by $f_{A,B,C,D}$ of geodesics defining them.
\end{lem}

\begin{proof}
 Consider any tree of these points, say $A, B$ and $C$. Due to Corollary \ref{generalus} there is a minimal triangular
 surface $S_{A,B,C}:T_{A,B,C}\to X$ spanned on these points. Denote by $\gamma_{AB}^{ABC}$ the geodesic between
 $A$ and $B$ from the minimal geodesic triangle associated to $S_{A,B,C}$. Similarly set  $\gamma_{AC}^{ABC}$ and 
 $\gamma_{BC}^{ABC}$. Similar minimal triangular surfaces and minimal geodesic triangles exist for
 every triple of points $A, B, C, D$.
 Denote by $A^{ABC}$ the last common vertex of $\gamma_{AB}^{ABC}$ and $\gamma_{AC}^{ABC}$, and similarly
 set $B^{ABC}$ and $C^{ABC}$. We use the similar notation for other triples.
 The general situation is presented in Figure \ref{6sysgen}, however some vertices can be identified.
 For example, geodesics $\gamma_{AB}^{ABD}, \gamma_{AB}^{ABD}$ can have more common vertices than endpoints
 (the same is true for any other pair of geodesics connecting the same pair of points), or
 geodesic segments $(A,A^{ABC})$ and triangles $(A^{ABC},B^{ABC}, C^{ABC})$ can be degenerated.
 All possible cases will be discussed later.
 Note that each point $A, B, C, D$ is contained in three convex sets (from the assumption), every geodesic in two convex 
 sets (from the convexity) and every triangular surface in one convex set (from Lemma \ref{conv}).
 
 \begin{figure}[ht!]
  \begin{center}
   \includegraphics[width=0.73\textwidth]{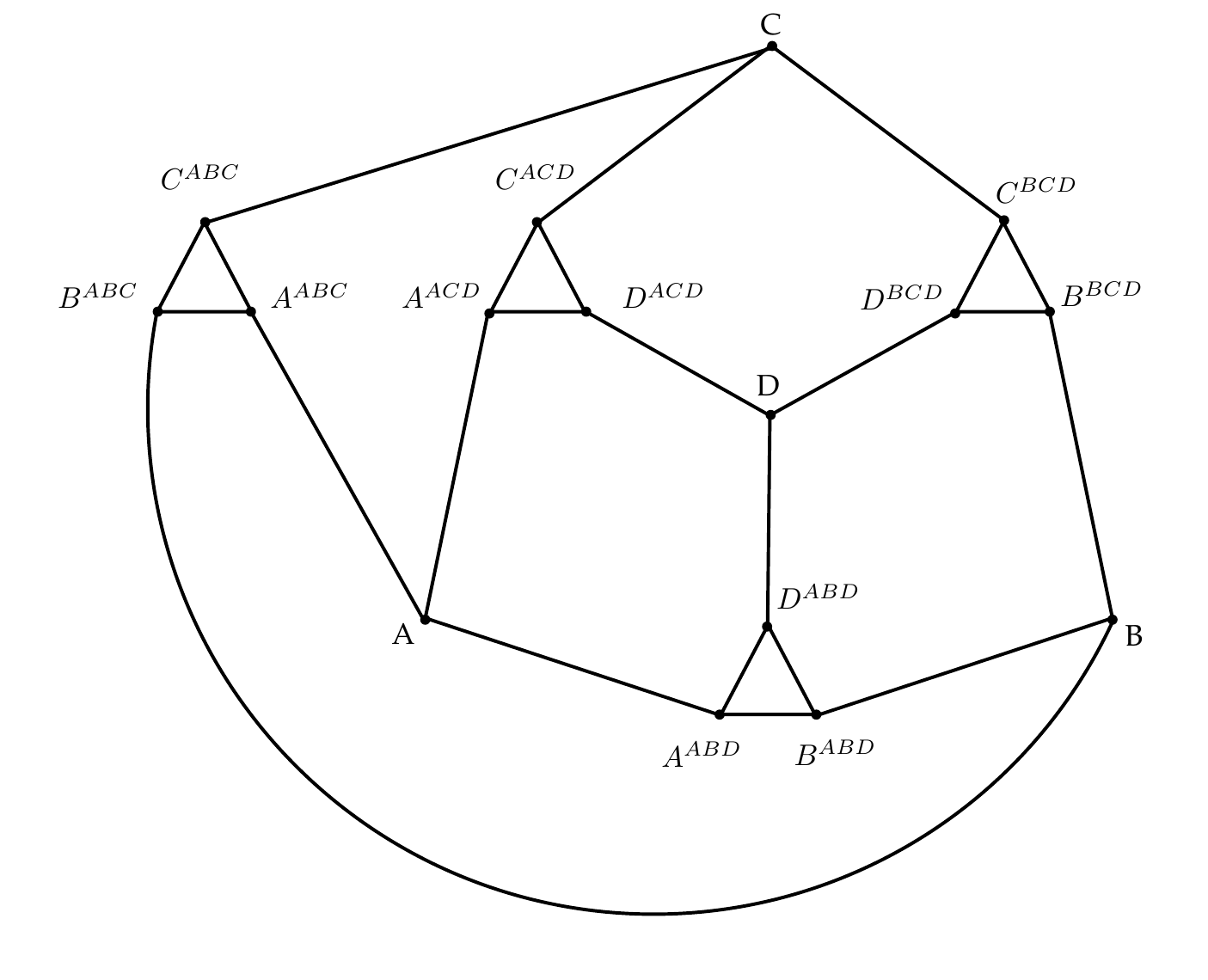}
   \caption{The general case.}\label{6sysgen}
  \end{center}
 \end{figure}
 
 Note that for two vertices $A, B$ there exist two (not necessary equal) geodesics $\gamma_{AB}^{ABC}$ and 
 $\gamma_{AB}^{ABD}$ connecting them. This defines the desired digonal surface 
 $S_{A, B}^{\gamma_{AB}^{ABD}, \gamma_{AB}^{ABD}}: D_{A,B}^{\gamma_{AB}^{ABD}, \gamma_{AB}^{ABD}} \to X$.
 For simplicity we denote this digonal surface by $S_{A, B}: D_{A,B} \to X$.
 We now modify the map $S_{A, B}$ to obtain a simplicial map, which has a triangulated disc as a domain.
 We do this similarly to the case presented in Figure \ref{didisc} and denote this map by
 $S'_{A, B}: D'_{A,B} \to X$.
 We repeat this construction for every pair of points.

 \begin{figure}[ht!]
  \begin{center}
   \includegraphics[width=0.2\textwidth,angle=270]{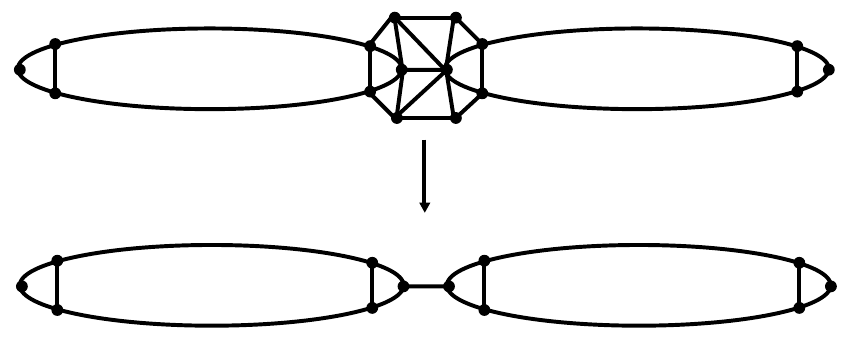}
   \caption{Extending $D_{A,B}^{\gamma_{AB}^{ABD}, \gamma_{AB}^{ABD}}$ to a disc.}\label{didisc}
  \end{center}
 \end{figure}

 Let $\tilde{S}$ be a disjoint union of $T_{u,v,w}$ and $D'_{u,v}$ and
 let $\tilde{f}_{A,B,C,D}: \tilde{S} \to X$ be defined as a respective digonal or triangular surface on each part of $\tilde{S}$.
 Define $S$ to be a quotient space $\tilde{S} /_{\sim}$, where $\sim$ is an equivalence relation identifying two boundary points
 of $T_{u,v,w}$ and $D'_{u,v}$ if they have the same image through $\tilde{f}_{A,B,C,D}$.
 We define $f_{A,B,C,D}$ to be a quotient of $\tilde{f}_{A,B,C,D}$. It follows from the construction that $f_{A,B,C,D}$ and $S$ satisfy
 demanded properties.
 
\end{proof}

We now prove the second main theorem of this paper:

\begin{tw}\label{tw6}
Let $X$ be a systolic complex and let $X_1, X_2, X_3, X_4$ be its convex subcomplexes such that every three of
them have a nonempty intersection. Then there exists a simplex $\sigma \subseteq X$ such that 
$\sigma \cap X_i \neq \emptyset$ for $i = 1,2,3,4$.
Moreover, $\sigma$ can be chosen to have the dimension at most three.
\end{tw}

\begin{proof}
 From the assumption there exist four points $A,B,C$ and $D$ such that $A \in X_1 \cap X_2 \cap X_3$,
 $B \in X_1 \cap X_2 \cap X_4$, $C \in X_1 \cap X_3 \cap X_4$, $D \in X_2 \cap X_3 \cap X_4$. 
 We choose these points in such a way that the sum of distances between them is minimal.
 First note that if any of these points lies on some geodesic connecting two other, then
 it is contained in all convex subcomplexes $X_1, X_2, X_3, X_4$ and therefore satisfies desired properties.
 Thus we can assume the opposite.

 Now we examine a neighbourhood of the vertex $D$. Set the notation as in the proof of Lemma \ref{spherus}.
 Suppose first that some of geodesic segments $(D, D^{ABD})$, $(D, D^{ACD})$ or $(D, D^{BCD})$ is nondegenerate and
 let $D'$ be a vertex from the link of $D$ lying on this segment. Without loss of generality we can assume that 
 $D'\in (D, D^{ABD})$. Since $(D, D^{ABD}) \subseteq \gamma_{AD} \cap \gamma_{BD}$,
 it holds that $D' \in X_2 \cap X_3 \cap X_4$. Note that $d(D',A)=d(D,A)-1$, $d(D',B)=d(D,B)-1$ and $d(D',C)\leq d(D,C)+1$.
 Thus concerning $D$ replaced with $D'$, we obtain another four points with the same properties as $A, B, C$ and $D$, but 
 with the smaller sum of distances between them. This contradicts the minimality of this sum for $A, B, C, D$.
 Therefore we can assume that all geodesic segments $(D, D^{ABD})$,  $(D, D^{ACD})$ and $(D, D^{BCD})$ are degenerated, i.e 
 $D=\allowbreak D^{ABD}= \allowbreak D^{ACD}=D^{BCD}$.
 
 Note that $D$ cannot be equal to any of points $A^{ACD}, A^{ABD}, B^{ABD}, B^{BCD}, C^{BCD}, C^{ACD}$
 (since it does not lay on any geodesic connecting $A$ to $B$, $A$ to $C$ or $B$ to $C$).
 Denote by $v^{ACD}_A$ the point closest to $D$ on the geodesic $(D, A^{ACD})$. Similarly denote set $v^{ABD}_A$, $v^{ABD}_B$,
 $v^{BCD}_B$, $v^{BCD}_C$ and $v^{ACD}_C$.
 Since $i_{A,D}^{-1}(v^{ABD}_A)$, $i_{A,D}^{-1}(v^{ACD}_A)$ and $i_{A,D}^{-1}(D)$ span $2$-simplex in $D_{A,D}$ and 
 $i_{A,D}$ is a simplicial map it follows that $v^{ABD}_A$ and $v^{ACD}_A$ are also connected if only $v^{ABD}_A \neq v^{ACD}_A$. 
 Same is true for pairs $v^{ABD}_B$, $v^{BCD}_B$ and $v^{BCD}_C$, $v^{ACD}_C$.
 Similarly $i_{A,B,D}^{-1}(v^{ABD}_A)$, $i_{A,B,D}^{-1}(v^{ABD}_B)$ and $i_{A,B,D}^{-1}(D)$ span $2$-simplex in $D_{A,B,D}$ and
 so $v^{ABD}_A$ and $v^{ABD}_B$ are connected by the fact that $i_{A,B,D}$ is a simplicial map and under the assumption that
 $v^{ABD}_A \neq v^{ABD}_B$.
 Same is true for pairs $v^{ACD}_A$, $v^{ACD}_C$ and $v^{BCD}_B$, $v^{BBD}_C$.
 If there are no identifications we end up in the situation presented on the Figure \ref{6kat'}.
  
  \begin{figure}[ht!]
   \begin{center}
    \includegraphics[width=0.4\textwidth]{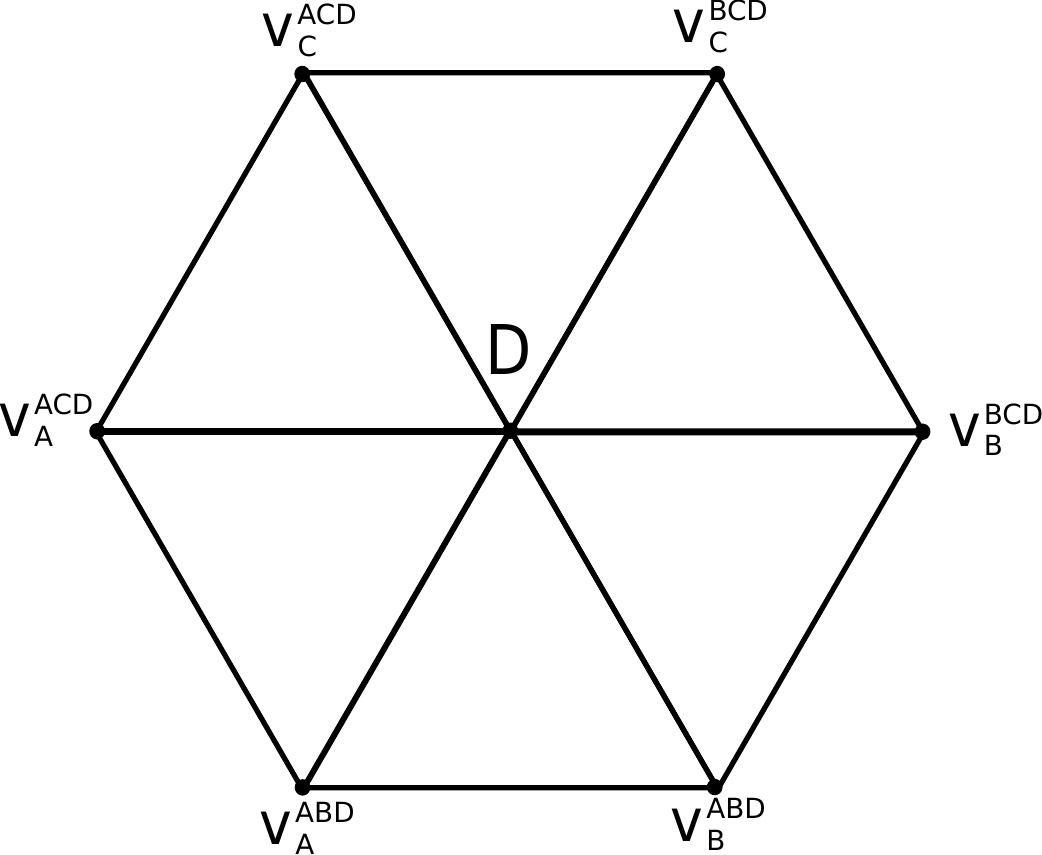}
    \caption{Part of the link of $D$ without any identification.}\label{6kat'}
   \end{center}
  \end{figure}
  
  We examine now possible identifications between these points.
  Note that $v^{ACD}_A \neq v^{ACD}_C$ (otherwise the geodesic segment $(D,D^{ACD})$ wouldn't be degenerated) and similarly 
  $v^{ABD}_A \neq v^{ABD}_B$ and $v^{BCD}_B \neq v^{BCD}_C$.
  If $v^{ACD}_A$ would be identified with any of these vertices, excluding $v^{ABD}_A$ or $v^{ACD}_C$, it would imply that
  $v^{ACD}_A \in X_2 \cap X_3 \cap X_4$. But then again replacing $D$ with $v^{ACD}_A$, we obtain another four points with 
  the same properties as $A, B, C, D$ and smaller sum of distances between them, which violates our assumption.
  Thus the only identifications left to concern are: $v^{ACD}_A$ with $v^{ABD}_A$, $v^{ABD}_B$ with $v^{BCD}_B$ and
  $v^{ACD}_C$ with $v^{BCD}_C$.
 
 Before proceeding with the proof of the theorem we show the following claim:
 

 \begin{claim}\label{link}
  If vertices $v^{ABD}_A$, $v^{ABD}_B$ are both connected with another vertex 
  $w \in \{v^{ACD}_A, \allowbreak v^{BCD}_B, \allowbreak v^{BCD}_C, \allowbreak v^{ACD}_C \}$, 
  then there exists a pairwise connected triple of points $ v_A, v_B, v_C$ in the link of $D$, such that $v_A \in (A,D)$,
  $v_B \in (B,D)$ and $v_C \in (C,D)$.
 \end{claim}

 \begin{proof}
  It is enough to show that there exist a pairwise connected triple of points $v_A \in \{ v^{ACD}_A, \allowbreak v^{ABD}_A \}$
  $v_B \in \{ v^{ABD}_B, v^{BCD}_B\}$ and $v_C \in \{v^{BCD}_C, \allowbreak v^{ACD}_C\}$. 
  If $w \in \{v^{BCD}_C, v^{ACD}_C \}$ the statement is trivially true. Otherwise we need to consider separately cases with the different 
  numbers of identifications between vertices
  $v^{ACD}_A, \allowbreak v^{ABD}_A, \allowbreak v^{ABD}_B, \allowbreak v^{BCD}_B, \allowbreak v^{BCD}_C, \allowbreak v^{ACD}_C$.
 
  \begin{enumerate}
 
   \item No identifications. If  $w = v^{ACD}_A$ then 
   $v^{ACD}_A, \allowbreak v^{ABD}_B, \allowbreak v^{BCD}_B, \allowbreak v^{BCD}_C, \allowbreak v^{ACD}_C$
   form a pentagon. Note that
   every pentagon in a systolic complex, contains a vertex $d$ which is connected by edges with its two opposite vertices.
   Then $v_A=v^{ACD}_A$, $v_B=v^{ABD}_B$, $v_C=v_C^{ACD}$   (if $d = v^{ABD}_B$ or $d = v^{ACD}_C$) or
   $v_A=v^{ACD}_A$, $v_B=v^{BCD}_B$, $v_C=v_C^{BCD}$   (if $d = v^{ACD}_A$) or
   $v_A=v^{ACD}_A$, $v_B=v^{ABD}_B$, $v_C=v_C^{BCD}$   (if $d = v^{BCD}_C$) or
   $v_A=v^{ACD}_A$, $v_B=v^{BCD}_B$, $v_C=v_C^{ACD}$   (if $d = v^{BCD}_B$).
   See Figure \ref{6kata} for an illustration.
   Case when $w=v^{BCD}_B$ is symmetric to the one described above.
   
   \begin{figure}[ht!]
    \begin{center}
     \includegraphics[width=0.8\textwidth]{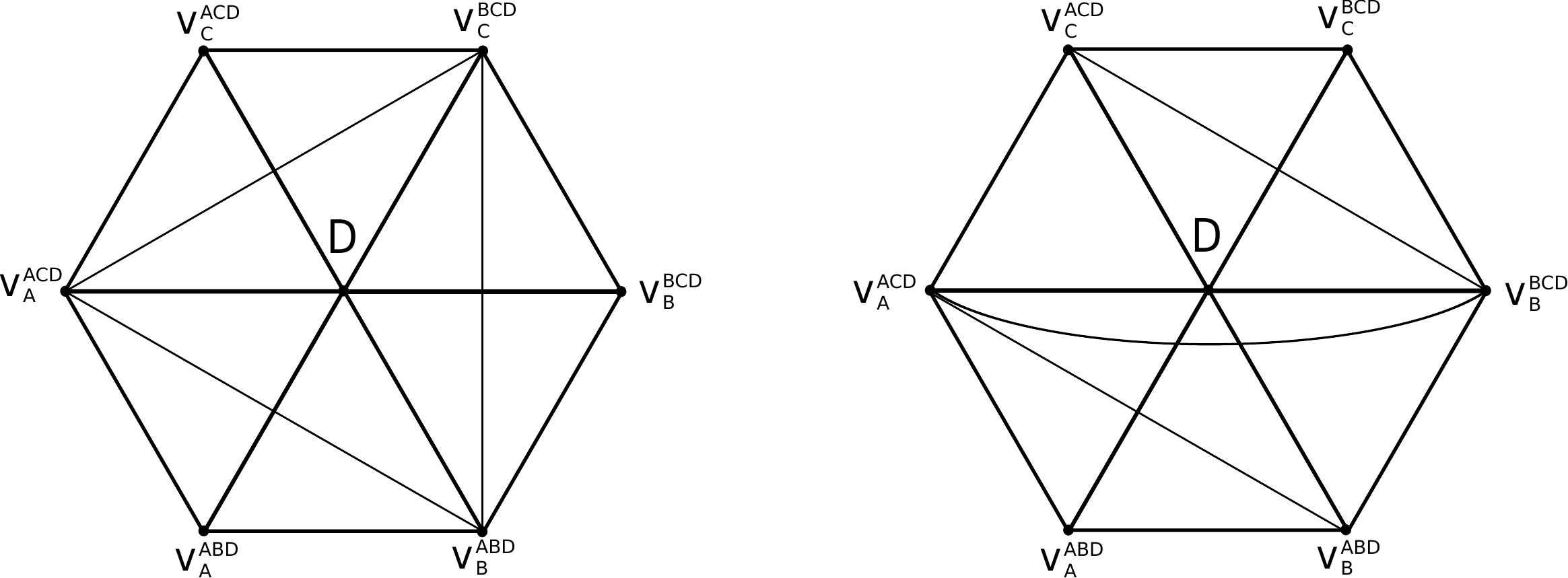}
     \caption{No identifications, $w = v^{ACD}_A$. On the left $d=v^{BCD}_C$. On the right $d=v^{BCD}_B$.}\label{6kata}
    \end{center}
   \end{figure}
   
   \item One identification. Assume that $v^{BCD}_C = v^{ACD}_C = v_C$, then 
   $v^{ACD}_A, \allowbreak v^{ABD}_A, \allowbreak v^{ABD}_B, \allowbreak v^{BCD}_B, \allowbreak v_C$ form a pentagon.
   Again we use the fact that every pentagon in a systolic complex, contains a vertex $d$ which is connected by edges with its two 
   opposite vertices.
   Then $v_A=v^{ACD}_A$, $v_B=v^{BCD}_B$ (if $d = v^{ACD}_A$ or $d = v^{BCD}_B$) or
   $v_A=v^{ABD}_A$, $v_B=v^{BCD}_B$ (if $d = v^{ABD}_A$) or
   $v_A=v^{ACD}_A$, $v_B=v^{ABD}_B$ (if $d = v^{ABD}_B$) or
   $v_A=v^{ABD}_A$, $v_B=v^{ABD}_B$ (if $d = v_C$)
   Cases when $v^{ABC}_A = v^{ABD}_A = v_A$ or $v^{BCD}_B = v^{ABD}_B = v_B$ are symmetric to the one described above.
    
   \item Two identifications. Assume $v^{ABC}_A = v^{ABD}_A = v_A$ and $v^{BCD}_B = v^{ABD}_B = v_B$. Then $v_A, \allowbreak v_B,
   \allowbreak v^{BCD}_C, \allowbreak v^{ACD}_C$ form a quadrilateral, which contains a diagonal from the fact that $X$ is systolic.
   If $v^{BCD}_C, v_A$ form a diagonal then $v_C = v^{BCD}_C$ and similarly if $v^{ACD}_C, v_B$ form a diagonal then $v_C = v^{ACD}_C$.
   Cases when $v^{ABC}_A = v^{ABD}_A = v_A, v^{BCD}_C = v^{ACD}_C = v_C$ or $v^{BCD}_B = v^{ABD}_B = v_B, v^{BCD}_C = v^{ACD}_C = v_C$
   are symmetric and allow the same reasoning.
   
   \item Three identifications. Then $v^{ABC}_A = v^{ABD}_A = v_A$, $v^{BCD}_B = v^{ABD}_B = v_B$, $v^{BCD}_C = v^{ACD}_C = v_C$
   and the statement of the claim is trivially satisfied.
   
  \end{enumerate}
 \end{proof}
    
  Now we are ready to give a crucial argument in this reasoning.
  Due to Lemma \ref{spherus} vertices $A, B, C, D$ determinate a simplicial map $f_{A,B,C,D}: S \to X$, where $S$ is a 
  triangulation of the $2$-sphere. 	
  From Lemma \ref{4.4} we know that this map can be extended to a simplicial map $F:\Delta \to X$, 
  where $\Delta$ is a triangulation of a $3$-ball such that $\partial \Delta = S$ and $\Delta$ has no internal vertices.
  Set $A':=F_{A,B,C,D}^{-1}(A)$ and similarly set $B', C', D'$.
  Also let $v'^{ACD}_A:=F_{A,B,C,D}^{-1}(v^{ACD}_A)$ and similarly set $v'^{ABD}_A, v'^{ABD}_B, v'^{BCD}_B$, $v'^{BCD}_C$, $v'^{ACD}_C$.
  Since $\Delta$ is a triangulation of a $3$-ball, there exists a vertex $w'$ spanning, together with 
  $D', v'^{ABD}_A, v'^{ABD}_B$ a $3$-simplex in $B$.
  Let $w= F_{A,B,C,D}(w')$. 
  We will examine  different positions of $w'$.
	
  Suppose first that $w' \in \partial D_{v,D}$ for some $v \in \{A,B,C\}$. 
  Then  $w' \in \{v'^{ACD}_A, \allowbreak v'^{BCD}_B, \allowbreak v'^{BCD}_C, \allowbreak v'^{ACD}_C\}$
  with the lower index $v$.
  Thus $w \in \{v^{ACD}_A, \allowbreak v^{BCD}_B, \allowbreak v^{BCD}_C, \allowbreak v^{ACD}_C\}$
  with the lower index $v$.
  By the Claim \ref{link} there exists a pairwise connected triple of points $v_A, v_B, v_C$ in the link of $D$, 
  such that $v_A \in (A,D)$, $v_A \in (B,D)$ and $v_C \in (C,D)$.
  If  $v_A \in X_1$ or $v_B \in X_1$ or $v_C \in X_1$ the theorem is proved as $v_A, v_B, v_C, D$ span a simplex 
  and $D \in X_2, X_3, X_4$. 
  Otherwise let $v'_A:=F_{A,B,C,D}^{-1}(v_A)$ and similarly set $v'_B, v'_C$.
  Since $v_A, v_B, v_C \notin X_1$ and $v'_A, v'_B, v'_C$ are not contained in the face of $\Delta$ which is the preimage of $X_1$
  there is a vertex opposite to $D'$.
  To be more precise there exist a vertex $d' \in \Delta$ different from $D'$, such that $v'_A, v'_B, v'_C, d'$ span a simplex. 
  Set $d:=F_{A,B,C,D}(d')$.
  Since $\Delta$ has no internal vertices $d \in X_i$ for some $i \in \{1,2,3,4\}$.
  Again if $d \in X_1$ the theorem is proved.
  Otherwise $d'$ is contained in a digonal or triangular surface associated to $D$.
  Without loss of generality assume that $d' \in D_{A,D}$.
  We will show that $d(D,d)=2$. 
  We assumed that $D' \neq d'$, so $d(D,d) \neq 0$.
  If $d(D',d')=1$ then $v'_A, v'_B, v'_C, d', D'$ span a $4$-simplex, which violates fact that 
  $\Delta$ is a triangulation of a $3$-ball.
  Thus $(D',v'_A, d')$ is a geodesic so $d(D',d')=2$ and $d(D,d) \leq 2$.
  Notice if $d(D,d)=0$ then image of $(A', D')$ is not a geodesic, a contradiction.
  If $d(D,d)=1$ and $d'$ is a boundary vertex of a digon, then going by the image of a geodesic $(A',d')$ and then by an edge $(d,D)$
  gives a path contradicting the distance between $A$ and $D$.
  If $d(D,d)=1$ and $d'$ is an internal vertex in $D_{A,D}$,
  there is another digonal surface $S'_{A,D}: D'_{A,D} \to X$ with smaller area then $D_{A,D}$.
  This surface is the restriction of $S_{A,D}$ to a digon between $A'$ and $d'$ and a geodesic segment from $d'$ to $D'$.
  A contradiction with the assumption that $S_{A,D}$ is a minimal digonal surface spanned on $A,D$.
  If $d'$ is contained in the triangular surface the similar reasoning applies.
  Thus we conclude that $d(D,d)=2$.
  Now let $v_d \in \{v_A, v_B, v_C \}$ be a vertex, which belongs to two convex subsets different from the one that $d$ belongs to.
  Since $(d,v_d,D)$ is a geodesic and by $3$-convexity $v_D \in X_2 \cap X_3 \cap X_4$. 
  But then replacing $D$ with $v_d$, we obtain another four points with 
  the same properties as $A, B, C, D$ and smaller sum of distances between them, which violates our assumption.
  
  Now consider $w' \in Int D_{v,D}$. If $w \in \{v^{ACD}_A, \allowbreak v^{BCD}_B, \allowbreak v^{BCD}_C, \allowbreak v^{ACD}_C\}$
  with the lower index $v$, then above reasoning applies.
  Otherwise $w'$ is an interior vertex of a simple digon contained in $D_{v,D}$. By Lemma \ref{digon} simple digons are 
  flat.
  Consider a digon in $\mathbb{R}^2_{\Delta}$, with the boundary consisting of two geodesics of length $n$. 
  Note that any interior vertex of this digon	is at distance at most $n-2$ from any of vertices spanning it.
  Thus if we apply $F_{A,B,C,D}$, we obtain a shortcut in a geodesic from $D$ to $v$, which gives a contradiction.
  
  Suppose now that $w' \in T_{v,u,D}$ for some $v,u \in \{A,B,C\}$. 
  If $w'$ is a vertex from a geodesic $(D',v')$ or $(D',u')$, similarly we have 
  $w \in \{v^{ACD}_A, \allowbreak v^{BCD}_B, \allowbreak v^{BCD}_C, \allowbreak v^{ACD}_C\}$
  and the previous reasoning applies.
  If $w'$ is at distance $2$ from $D'$ (there is only one such an internal vertex in $T_{v,u,D}$),
  there is another triangular surface $S'_{v,u,D}: T'_{v,u,D} \to X$ with smaller area then $S_{v,u,D}$.
  This surface is the restriction of $S_{v,u,D}$ to a domain bounded by $(D',w') \cup (w',v') \cup (w',u') \cup (v',u')$.	
  A contradiction with the assumption that $T_{v,u,D}$ is a minimal triangular surface spanned on $v,u,D$.
  If $w'$ is in distance greater than $2$ from $D'$, consider paths $(D,w) \cup (w,v)$ and $(D,w) \cup (w,u)$. 
  It follows from the triangle inequality for $T_{v,u,D}$ that one of these paths is shorter than the respective geodesic. 
  A contradiction.
  
  Finally we consider what happen either $w' \in T_{A,B,C}$ or $w' \in D_{v,u}$ for some $v,u \in \{A,B,C\}$.
  Since $A,B,C \in X_1$, due to Lemma \ref{conv} we obtain that $w \in X_1$.
  Thus vertices $D, v^{ABD}_A, v^{ABD}_B, w$ span a $3$-simplex with desired properties.
	
\end{proof}

\end{document}